\documentclass{birkmult}

 \newtheorem{thm}{Theorem}[section]
 
 \newtheorem{lem}[thm]{Lemma}
 
 \theoremstyle{definition}
 \newtheorem{defn}[thm]{Definition}
 \theoremstyle{remark}
 \newtheorem{rem}[thm]{Remark}
 
 \numberwithin{equation}{section}

 \usepackage{amssymb}
 \def\P{\mathrm{P}}
 \def\L{\mathrm{L}}
 \def\x{\mathrm{x}}
 
 \def\a{\mathrm{a}}
 \def\A{\mathrm{A}}
 
 \def\p{\mathrm{p}}
 
 \def\s{\mathrm{s}}
 \def\t{\mathrm{t}}
 \def\c{\mathrm{c}}
 \def\f{\mathrm{f}}

 \def\RR{\mathbb R}

\begin{document}
%
%
%
%
%
%
%
%
%
\title[On solutions of elliptic equations with variable exponents]
 {On solutions of elliptic equations with\\ variable exponents and measure data in $\mathbb{R}^n$}
\author[L.\,M.~Kozhevnikova]{L.\,M.~Kozhevnikova}

\address{%
Sterlitamak Branch of Bashkir State University\\
Sterlitamak, Russia;\\
Elabuga Institute of Kazan Federal\\
University, Elabuga, Russia}

\email{kosul@mail.ru}

\thanks{This work was supported by the RFBR under Grant [18-01-00428].}


\date{February 1, 2018}

\begin{abstract}
	Anisotropic elliptic equations of the second order with variable exponents in nonlinearities and the right-hand side as a diffuse measure are considered in the space $\mathbb{R}^n$.
	The existence of an entropy solution in anisotropic Sobolev--Orlicz spaces with variable exponents is proved. The obtained entropy solution is shown to be a renormalized solution.
\end{abstract}

\maketitle
\section{Introduction}\label{s1}

In the space
$\mathbb{R}^n=\{\x=(x_1,x_2,\ldots,x_n)\}$, $n \geq 2,$ we consider the anisotropic elliptic equation
\begin{equation}\label{ur}
-{\rm div}\,\a(\x,u,\nabla
u)+|u|^{p_0(\x)-2}u+b(\x,u,\nabla u)=\mu,\quad \x\in \mathbb{R}^n,
\end{equation}
where $\mu$ is a bounded Radon measure of a special form.

The concept of renormalized solutions serves as the main step in the investigation of general degenerate elliptic equations with a measure data.
The first definition of such solutions has been given in \cite{DMOP_1999}, \cite{Malusa}, and then extended in \cite{Bidaut_Veron} to a local form.
The key result in \cite{DMOP_1999} is the claim that each Radon measure $\mu$ in a bounded domain $\Omega\subset \mathbb{R}^n$ with finite total variation can be decomposed as $\mu =
\mu_0 + \mu_s$.
Here, $\mu_0$ is called a diffuse measure with respect to the $p$--capacity and $\mu_0(E)=0$ for every $E\subseteq\Omega$ such that ${\rm Cap}_{p}\,(E,\Omega)=0$, while the measure $\mu_s$ is called singular and it is concentrated on a set of the $p$--capacity zero.
It was shown that $\mu_0$ is diffuse with respect to the $p$--capacity if and only if $\mu_0 \in L_1 (\Omega) + H^{-1}_{p'} (\Omega)$ ($H^{-1}_{p'} (\Omega)$ is the dual space to $\mathring
{H}_{p}^1(\Omega)$), i.e., there exist functions $f\in L_1(\Omega)$ and $\f=(f_1,\dots,f_n)\in (L_{p'}(\Omega))^n$ such that
\begin{equation}\label{DMOP}
\mu_0=f-{\rm div}\, \f.
\end{equation}
The stability and existence of a renormalized solution of an equation of the type
\begin{equation}\label{ur0}
-{\rm div}\,\a(\x,\nabla
u)+a_0(\x,u)=\mu,\quad \x\in \Omega,
\end{equation}
where $a_0(\x,s_0)\equiv 0$ and $\Omega$ is a bounded domain, under the homogeneous Dirichlet boundary conditions
\begin{equation}\label{gu}
u = 0 ~~\text{on}~~ \partial \Omega,
\end{equation}
have been studied in the works \cite{DMOP_1999}, \cite{Malusa}.

 In \cite{BMMP_2002} it was proved that,
again when $\Omega$
 is bounded but $a_0\not \equiv 0,$ for every Radon measure $\mu$ with bounded variation in $\Omega$  there exists at least a
renormalized solution of the Dirichlet problem \eqref{ur0}, \eqref{gu}. Moreover in \cite{BMMP_2002_}  it is shown that such a solution is unique whenever $\mu$ does not charge sets of negligible $p$-capacity.

It is shown in \cite{Malusa_Porzio} that, in the case of an arbitrary domain $\Omega \subsetneq \mathbb{R}^n$, $\mu_0$ is diffuse with respect to the $p$--capacity if and only if $\mu_0\in L_1 (\Omega) + W^{-1}_{p'} (\Omega)$ ($W^{-1}_{p'} (\Omega)$ is the dual space to $\mathring{W}_{p}^1(\Omega)$), that is, there exist functions $f\in L_1(\Omega),\;f_0\in L_{p'}(\Omega)$, $\f=(f_1,\dots,f_n)\in (L_{p'}(\Omega))^n$ such that
\begin{equation}\label{malusa}
\mu_0=f+f_0-{\rm div}\, \f.
\end{equation}
In \cite{Malusa_Porzio}, the existence, and for $\mu = \mu_0$ also the uniqueness, of a renormalized solution of the Dirichlet problem \eqref{ur0}, \eqref{gu} in an arbitrary domain $\Omega$ have been obtained.
Let us note that the $p$--capacity in the articles \cite{DMOP_1999}, \cite{Malusa_Porzio} is defined differently, which is due to the choice of the main space ($\mathring{H}_{p}^1(\Omega)$ or $\mathring
{W}_{p}^1(\Omega)$). This leads to the difference in the representation of the diffuse measure (see \eqref{DMOP}, \eqref{malusa}).

The modern theory of elliptic equations with non-standard growth conditions was developed by V. V.\ Zhikov \cite{Zhik}, Yu. A.\ Alkhutov \cite{Alxut}.
The most important variational methods for elliptic PDEs described by nonhomogeneous differential operators and containing one or more power-type nonlinearities with variable exponents are presented in the monograph \cite{Repov0}. The authors of \cite{Repov0} give a systematic treatment of the basic mathematical theory and constructive methods for these classes of nonlinear elliptic equations as well as their applications to various processes arising in the applied sciences.

Let $Q\subseteq\mathbb{R}^n$ be an arbitrary domain in $\mathbb{R}^n$.
Introduce the space
$$
C^+(\overline Q)=\{p\in C(\overline Q)\;:\; 1<p^-\leq p^+<+\infty \},
$$
where $p^-=\inf\limits_{\x\in Q}p(\x)$, $p^+=\sup\limits_{\x\in Q}p(\x)$.
For
$p(\cdot)\in C^+(\overline{Q})$, we define the Lebesgue space with variable exponent
$L_{p(\cdot)}(Q)$ as a set of measurable on $Q$ real-valued functions $v$ such that
$$
\rho_{p(\cdot),Q}(v)=\int\limits_Q|v(\x)|^{p(\x)}d\x<\infty.
$$
The Luxemburg norm of $L_{p(\cdot)}(Q)$ is defined by
$$
\|v\|_{L_{p(\cdot)}(Q)}=\|v\|_{p(\cdot),Q}=\inf\left\{k>0\;\Big |\;
\rho_{p(\cdot),Q}(v/k)\leq1\right\}.
$$
The space $L_{p(\cdot)}(Q)$ is a separable reflexive Banach space \cite{Diening}.

The norm of $L_p(Q)$, $p\in[1,\infty],$ will be denoted as $\|\cdot\|_{p,Q}$.
In the case where $Q=\mathbb{R}^n$, we will use the corresponding notations
$C^+(\mathbb{R}^n)$, $\|\cdot\|_{p(\cdot)}$, $\rho_{p(\cdot)}(\cdot)$, $\|\cdot\|_{p}$.

For an arbitrary domain $\Omega\subsetneq \mathbb{R}^n$, we define the Sobolev space with variable exponent  $\mathring{H}_{p(\cdot)}^1(\Omega)$ as a completion of $C_0^{\infty}(\Omega)$ with respect to the norm
$$
\|v\|_{\mathring{H}_{p(\cdot)}^1(\Omega)}=\|\nabla v\|_{p(\cdot),\Omega}.
$$

The set of bounded Radon measures will be denoted as $\mathcal{M}^b(\Omega)$.
A measure $\mu\in\mathcal{M}^b(\Omega)$ is called diffuse with respect to the capacity ${\rm Cap}_{p(\cdot)}$ ($p(\cdot)$--capacity), whenever $\mu(E)=0$ for any $E$ such that ${\rm Cap}_{p(\cdot)}\,(E,\Omega)=0$.

Let $\mathcal{M}^b_{p(\cdot)}(\Omega)$ stand for the space of all bounded Radon measures which are diffuse with respect to the $p(\cdot)$--capacity.
It is proved in \cite{Ouaro_2013}, \cite{Zhang_2014} that, in the case of a bounded domain $\Omega$,  $\mu\in\mathcal{M}^b_{p(\cdot)}(\Omega)$ if and only if $\mu\in L_1(\Omega)+H^{-1}_{p'(\cdot)}(\Omega)$ ($H^{-1}_{p'(\cdot)} (\Omega)$ is the dual space to $\mathring{H}_{p(\cdot)}^1(\Omega)$), that is, there exist functions $f\in L_1(\Omega)$, $\f=(f_1,\dots,f_n)\in (L_{p'(\cdot)}(\Omega))^n$ such that
$$
\mu=f-{\rm div}\, \f.
$$

Assuming that $\Omega\subset \mathbb{R}^n$ is bounded, the authors of \cite{Zhang_Zhou} considered the Dirichlet problem for the equation with the $p(\x)$-Laplacian
\begin{equation*}
-{\rm div}\,(|\nabla
u|^{p(\x)-2}\nabla
u)=\mu,\quad \x\in \Omega,
\end{equation*}
under the boundary conditions \eqref{gu}.
The existence and uniqueness of entropy and renormalized solutions, as well as their equivalence, have been proved for $\mu\in L_1(\Omega)+H^{-1}_{p'(\cdot)}(\Omega)$.

The existence and uniqueness of an entropy solution of the Dirichlet problem in a bounded domain $\Omega\subset \mathbb{R}^n$, $n\geq 3,$ for the equation
\begin{equation*}
-\sum_{i=1}^n (a_i(\x,
u_{x_i}))_{x_i}+b(u)=\mu,\quad \x\in \Omega,
\end{equation*}
under the boundary conditions \eqref{gu} have been considered in \cite{Konate_Ouaro}.

Let us mention the work \cite{Afr_2018} which studies the existence of a renormalized solution of the Dirichlet problem in a bounded domain $\Omega$ for an equation with variable exponents in nonlinearities of the type \eqref{ur0} with $a_0(\x,s_0) \equiv 0$ and a general Radon measure of finite total variation.

The articles \cite{Benboubker_2015}, \cite{Azroul_2018} are devoted to the existence of entropy solutions of the Dirichlet problem in a bounded domain $\Omega$ for equations with variable exponents in nonlinearities of the type
\begin{equation}\label{urf}
-{\rm div}\,(\a(\x,u,\nabla
u)+\c(u))+a_0(\x,u,\nabla
u)=\mu,\quad \x\in \Omega.
\end{equation}
The existence of an entropy solution of problem \eqref{urf}, \eqref{gu} with  $\mu=f-{\rm div}\, \f\in L_1(\Omega)+H^{-1}_{p'(\cdot)}(\Omega)$ has been established.
In \cite{Benboubker_2015}, \cite{Azroul_2018}, it is assumed that
$\c\in C_0(\mathbb{R},\mathbb{R}^n)$, $\a(\x,s_0,\s):\Omega\times\mathbb{R}\times\mathbb{R}^n\rightarrow \mathbb{R}^n$ is a Carath\'eodory function, and there exist a nonnegative function $\Phi\in L_{p'(\cdot)}(\Omega)$ and positive numbers $\widehat{a}, \overline{a}$ such that for a.e.\ $\x\in\Omega$ and all $s_0\in \mathbb{R}$,  $\s,\t\in\mathbb{R}^n$, $i=1,\ldots,n$,  there hold
\begin{equation}\label{ogr}
|\a(\x,s_0,\s)|\leq \widehat{a}\left(|s_0|^{p(\x)-1}+|\s|^{p(\x)-1}
+\Phi(\x)\right);
\end{equation}
\begin{equation}\label{mon}
(\a(\x,s_0,\s)-\a(\x,s_0,\t))\cdot(\s-\t)>0,\quad \s\neq \t;
\end{equation}
\begin{equation}\label{koerc}
\a(\x,s_0,\s)\cdot\s\geq \overline{a}|\s|^{p(\x)}.
\end{equation}
Here, ${\s}\cdot{\t}=\sum\limits_{i=1}^ns_it_i,\;$$\s=(s_1,\ldots,s_n),\;$$\t=(t_1,\ldots,t_n)$.

Let us remark that, introducing the notation $\widetilde{\a}(\x,s_0,\s)=\a(\x,s_0,\s)+\f$, we get the equation
$$
-{\rm div}(\widetilde{\a}(\x,u,\nabla
u)+\c(u))+a_0(\x,u,\nabla u)=f
$$
where the function $\widetilde{\a}(\x,s_0,\s)$ satisfies assumptions \eqref{ogr}, \eqref{mon}.
Moreover, the coercivity assumption \eqref{koerc} is written as
$$
\widetilde{\a}(\x,s_0,\s)\cdot\s\geq \widetilde{a}|\s|^{p(\x)}-\phi(\x),\quad \phi\in L_1(\Omega).
$$

In the present article, this idea is developed for the anisotropic equation \eqref{ur} with a diffuse measure $\mu$ and a function $\a(\x,s_0,\s)$ on which less restrictive assumptions than in \cite{Benboubker_2015}, \cite{Azroul_2018} are imposed (see assumptions \eqref{us1}--\eqref{us3} below).

Thus, the works known to the author contain results on entropy and renormalized solutions of elliptic problems in bounded domains (except \cite{Beni}, \cite{Bidaut_Veron}, \cite{Malusa_Porzio} for equations with power-type nonlinearities).
In the present article, the existence of entropy solutions of equation \eqref{ur} in the whole $\mathbb{R}^n$ from anisotropic Sobolev--Orlicz spaces with variable exponents is established.
Moreover, it is proved that the obtained solution is a renormalized solution of equation \eqref{ur}.

Notice that results on the existence and uniqueness of entropy and renormalized solutions of nonlinear equations with variable exponents in nonlinearities and right-hand side from the space $L_1(\Omega)$ in arbitrary domains have been obtained in \cite{Kozhe_SMFN}--\cite{Muk}.

\section{Capacity, measure, and anisotropic Sobolev--Orlicz spaces with variable exponents}\label{s2}

Let
$p\in C^+(\mathbb{R}^n)$.
The following Young's inequality is satisfied:
\begin{equation}\label{ung}
|zy|\leq |y|^{p(\x)}+|z|^{p'(\x)},\quad z,y\in \mathbb{R},\quad \x\in \mathbb{R}^n, \quad p'(\x)=p(\x)/(p(\x)-1).
\end{equation}
Moreover, in view of convexity, there holds
\begin{equation}\label{sum}
|y+z|^{p(\x)}\leq 2^{p^+-1}(|y|^{p(\x)}+|z|^{p(\x)}),\quad z,y\in \mathbb{R},\quad\x\in \mathbb{R}^n.
\end{equation}
For any $u\in L_{p'(\cdot)}(\mathbb{R}^n)$, $v\in L_{p(\cdot)}(\mathbb{R}^n)$, the H\"older inequality
\begin{equation}\label{gel}
\int\limits_{\mathbb{R}^n} |u(\x)v(\x)|\, d\x\leq 2\|u\|_{p'(\cdot)}\|v\|_{p(\cdot)}
\end{equation}
is valid, and the following relations take place~\cite{Diening}:
\begin{equation*}
\min\{\|v\|^{p^-}_{p(\cdot)},\|v\|^{p^+}_{p(\cdot)}\}\leq\rho_{p(\cdot)}(v)\leq\max\{\|v\|^{p^-}_{p(\cdot)},\|v\|^{p^+}_{p(\cdot)}\}\leq \|v\|^{p^+}_{p(\cdot)}+1,\label{st2}
\end{equation*}
\begin{align}
\notag
\min\{\rho^{1/p^+}_{p(\cdot)}(v),\rho^{1/p^-}_{p(\cdot)}(v)\}
\leq \|v\|_{p(\cdot)}
&\leq \max\{\rho^{1/p^+}_{p(\cdot)}(v),\rho^{1/p^-}_{p(\cdot)}(v)\}\\
&\leq\left(\rho_{p(\cdot)}(v)+1\right)^{1/p^-}.
\label{st3}
\end{align}

Assume $\overrightarrow{\p}(\cdot)=(p_1(\cdot),p_2(\cdot),\ldots,p_n(\cdot))\in (C^+(\mathbb{R}^n))^n$, $\overrightarrow{\bf p}(\cdot)=(p_0(\cdot),\overrightarrow{\p}(\cdot))\in (C^+(\mathbb{R}^n))^{n+1}$, and define
$$
p_+(\x)=\max_{i=\overline{1,n}}p_i(\x),\quad p_-(\x)=\min_{i=\overline{1,n}}p_i(\x),\quad  \x\in \mathbb{R}^n.
$$
Let $\overrightarrow{\bf p}(\cdot)=(p_0(\cdot),p_1(\cdot),\dots,p_n(\cdot))\in (C^+(\mathbb{R}^n))^{n+1}$ and let $\Omega\subsetneq\mathbb{R}^n$ be an arbitrary domain in $\mathbb{R}^n$.
Anisotropic Sobolev--Orlicz spaces with variable exponents $\mathring{H}_{\overrightarrow{\p}(\cdot)}^1(\Omega)$, ${W}_{\overrightarrow{\bf p}(\cdot)}^{1}(\mathbb{R}^n)$ are defined as completions of the spaces $C_0^{\infty}(\Omega)$, $C_0^{\infty}(\mathbb{R}^n)$ with respect to the norms
\begin{align*}
\|v\|_{\mathring{H}_{\overrightarrow{\p}(\cdot)}^1(\Omega)}
&=\sum_{i=1}^n\|v_{x_i}\|_{p_i(\cdot),\Omega},\\
\|v\|_{{W}_{\overrightarrow{\bf p}(\cdot)}^{1}(\mathbb{R}^n)}
&=\|v\|_{p_0(\cdot)}+\sum_{i=1}^n\|v_{x_i}\|_{p_i(\cdot)},
\end{align*}
respectively.
The spaces  $\mathring{H}_{\overrightarrow{\p}(\cdot)}^1(\Omega)$, ${W}_{\overrightarrow{\bf p}(\cdot)}^{1}(\mathbb{R}^n)$ are reflexive Banach spaces \cite{Fan}.

We define by ${\L}_{\overrightarrow{\p}(\cdot)}(\mathbb{R}^n)$ the space $L_{p_1(\cdot)}(\mathbb{R}^n)\times\ldots\times L_{p_n(\cdot)}(\mathbb{R}^n)$ endowed with the norm
$$
\|{\rm v}\|_{{\L}_{\overrightarrow{\p}(\cdot)}(\mathbb{R}^n)}=\|{\rm v}\|_{\overrightarrow{\p}(\cdot)}=\|v_1\|_{p_1(\cdot)}+\ldots+\|v_n\|_{p_n(\cdot)},\quad
{\rm v}=(v_1,\ldots,v_n)\in {\L}_{\overrightarrow{\p}(\cdot)}(\mathbb{R}^n).
$$

The set of bounded Radon measures will be denoted as $\mathcal{M}^b(\mathbb{R}^n)$.
A measure $\mu\in\mathcal{M}^b(\mathbb{R}^n)$ will be called diffuse with respect to the capacity ${\rm Cap}_{\overrightarrow{\bf p}(\cdot)}$ ($\overrightarrow{\bf p}(\cdot)$--capacity) if $\mu(B)=0$ for any Borel set $B\subset\mathbb{R}^n$ such that ${\rm Cap}_{\overrightarrow{\bf p}(\cdot)}\,(B,\mathbb{R}^n)=0$.
Here, the $\overrightarrow{\bf p}(\cdot)$--capacity of a compact  set $K$ with respect to $\mathbb{R}^n$ is defined by
$$
{\rm Cap}_{\overrightarrow{\bf p}(\cdot)}\,(K,\mathbb{R}^n)=\inf\limits_{S_{p(\cdot)}(K)}\int\limits_{\mathbb{R}^n}{\bf P}(\x,v,\nabla v)\,d\x,
$$
$$
S_{p(\cdot)}(K)=\left\{v\in
C^{\infty}_0(\mathbb{R}^n)\;\Big |\; v\geq \chi_K  \right\},
$$
where
${\bf P}(\x,s_0,\s)=\P(\x,\s)+|s_0|^{p_0(\x)}$, $\P(\x,\s)=\sum\limits_{i=1}^n|s_i|^{p_i(\x)}$, and
$\chi_K$ is the characteristic function of the set $K$.
Then the $\overrightarrow{\bf p}(\cdot)$--capacity of a Borel set $B\subset \mathbb{R}^n$ with respect to $\mathbb{R}^n$ is defined by
$$
{\rm Cap}_{\overrightarrow{\bf p}(\cdot)}\,(B,\mathbb{R}^n)=\sup\left\{{\rm Cap}_{\overrightarrow{\bf p}(\cdot)}\,(K,\mathbb{R}^n)\;\Big |\;B\supset K,\;K \;\mbox{is a compact set K}\right\}.
$$
The space of all bounded Radon measures which are diffuse with respect to the $\overrightarrow{\bf p}(\cdot)$--capacity will be denoted as $\mathcal{M}^b_{\overrightarrow{\bf p}(\cdot)}(\mathbb{R}^n)$.

\begin{lem}\label{lemma-1}
	If $\mu\in L_1(\mathbb{R}^n)+W^{-1}_{\overrightarrow{\bf p}'(\cdot)}(\mathbb{R}^n)$, where $W^{-1}_{\overrightarrow{\bf p}'(\cdot)}(\mathbb{R}^n)$ is the dual space to $W_{\overrightarrow{\bf p}(\cdot)}^{1}(\mathbb{R}^n)$,
	then $\mu\in\mathcal{M}^b_{\overrightarrow{\bf p}(\cdot)}(\mathbb{R}^n)$.
\end{lem}
\begin{proof}
	We perform the proof for a positive measure $\mu$.
	If $\mu\in L_1(\mathbb{R}^n)+W^{-1}_{\overrightarrow{\bf p}'(\cdot)}(\mathbb{R}^n)$, then there exist functions $f$, $f_0$, $\f$ such that
	\begin{equation}\label{us7}
	\mu=f+f_0-{\rm div}\, \f,\quad f\in L_1(\mathbb{R}^n),\quad f_0\in L_{p'_0(\cdot)}(\mathbb{R}^n),\quad \f=(f_1,\dots,f_n)\in {\L}_{\overrightarrow{\p}'(\cdot)}(\mathbb{R}^n).
	\end{equation}
	Consider a Borel subset $B\subset\mathbb{R}^n$ such that ${\rm Cap}_{\overrightarrow{\bf p}(\cdot)}\,(B,\mathbb{R}^n)=0$ and let $K\subset B$ be an arbitrary  compact set. Then ${\rm Cap}_{\overrightarrow{\bf p}(\cdot)}\,(K,\mathbb{R}^n)=0$.
	By definition, there exists a sequence $\{\phi^j\}\subset C^{\infty}_0(\mathbb{R}^n)$ satisfying $\phi^j\geq \chi_K$ and $\int\limits_{\mathbb{R}^n}{\bf P}(\x,\phi^j,\nabla \phi^j)\, d\x\rightarrow 0$ as $j\rightarrow\infty$.
	Assume $v^j=T_1\phi^j$. Then $0\leq v^j\leq 1$, $v^j=1$ in $K$, and $\int\limits_{\mathbb{R}^n}{\bf P}(\x,v^j,\nabla v^j)\, d\x\rightarrow 0$ as $j\rightarrow\infty$.
	This implies that $\|v^j\|_{p_0(\cdot)}+
	\|\nabla v^j\|_{\overrightarrow{\p}(\cdot)}\rightarrow 0$ as $j\rightarrow\infty$.
	
	For the positive measure $\mu$, we get
	\begin{align}
	\notag
	\mu(K)=\int\limits_K
	d\mu\le\int\limits_{R^n}v_jd\mu
	&=
	\int\limits_{R^n} (f+f_0-{\rm div}\;\f)v^j\, d\x
	\\
	\label{mu}
	&\leq
	\int\limits_{\mathbb{R}^n} (|f|v^j+|f_0|v^j+|\f \cdot \nabla v^j|)\, d\x.
	\end{align}
	
	Let
	$$
	f^m(\x)=T_m(f)\chi_{ B(m)},\quad   B(m)=\{\x\in \mathbb{R}^n:|\x|<m\}, \quad m\in\mathbb{N}.
	$$
	It is not hard to see that
	\begin{equation}\label{theorem_2(1-0)}
	f^m\rightarrow f \quad \mbox {in}  \quad L_1(\mathbb{R}^n),\quad m\rightarrow\infty,
	\end{equation}
	and
	\begin{equation}\label{theorem_2(1-1)}
	|f^m(\x)|\leq |f(\x)|,\quad |f^m(\x)|\leq m\chi_{ B(m)},\quad \x\in \mathbb{R}^n,\quad m\in\mathbb{N}.
	\end{equation}
	Then, applying \eqref{gel}, we derive from \eqref{mu} the following inequality:
	$$
	\mu(K)\leq \|v^j\|_{\infty}\|f-f^m\|_{1}+2(\|f^m\|_{p'_0(\cdot)}+\|f_0\|_{p'_0(\cdot)})\|v^j\|_{p_0(\cdot)}+
	2\|\f\|_{\overrightarrow{\p}'(\cdot)}\|\nabla v^j\|_{\overrightarrow{\p}(\cdot)}.
	$$
	First, taking the limit as $j\rightarrow\infty$, we obtain
	$$
	\mu(K)\leq \|f-f^m\|_{1}.
	$$
	Second, taking the limit as $m\rightarrow\infty$, we get $\mu(K)=0$.
	Therefore,
	$$
	\mu(B)=\sup\{\mu(K)\;\Big |\;B\supset K,\;K \;\mbox{is a compact set}\}=0.
	\eqno\qedhere
	$$
\end{proof}

Let us state an embedding theorem for the space $\mathring{H}_{\overrightarrow{\p}(\cdot)}^1(\Omega)$, see \cite[Theorem 2.5]{Fan}.
\begin{lem}\label{lemma_0}
	Let $\Omega$ be a bounded domain, $\overrightarrow{\p}(\cdot)=(p_1(\cdot),\dots,p_n(\cdot))\in (C^+(\overline{\Omega}))^n$, and
	$$
	\overline{p}(\x)={n}\left(\sum\limits_{i=1}^n 1/p_i(\x)\right)^{-1},\quad
	p_*(\x)=\left\{\begin{array}{ll}\frac{n\overline{p}(\x)}{n-\overline{p}(\x)},& \overline{p}(\x)<n,\\
	+\infty,& \overline{p}(\x)\geq n,
	\end{array}\right.
	$$
	$$
	p_{\infty}(\x)=\max\{p_*(\x),p_+(\x)\},\quad \x\in \Omega.
	$$
	If $q\in C^+(\overline{\Omega})$ and
	\begin{equation*}\label{kozhelm-1}
	q(\x)<p_{\infty}(\x)\quad \forall\;\x\in \Omega,
	\end{equation*}
	then the embedding $\mathring	{H}_{\overrightarrow{\p}(\cdot)}^1(\Omega)\hookrightarrow L_{q(\cdot)}(\Omega)$ is continuous and compact.
\end{lem}

\section{Assumptions and main results}\label{s3}

We will suppose that
\begin{equation}\label{us0}
p_+(\x)\leq p_0(\x),\quad \x\in \mathbb{R}^n.
\end{equation}
It is assumed that the functions
$$
{\a}(\x,s_0,\s)=(a_1(\x,s_0,\s),\ldots,a_n(\x,s_0,\s)):\mathbb{R}^n\times\mathbb{R}\times\mathbb{R}^n\rightarrow \mathbb{R}^n,
$$
$$
b(\x,s_0,\s):\mathbb{R}^n\times\mathbb{R}\times\mathbb{R}^n\rightarrow \mathbb{R},
$$
from equation \eqref{ur} are measurable with respect to $\x\in \mathbb{R}^n$ for all $s_0\in\mathbb{R},\; \mathrm{s}=(s_1,\ldots,s_{n})\in\mathbb{R}^{n}$, and continuous with respect to $(s_0,\mathrm{s})\in\mathbb{R}^{n+1}$ for a.e.\ $\x\in\mathbb{R}^n$.
Assume that there exist nonnegative functions $\Phi_i\in L_{p'_i(\cdot)}(\mathbb{R}^n)$, $\phi\in L_1(\mathbb{R}^n)$, continuous nondecreasing functions $\widehat{a}_i:\mathbb{R}^+\rightarrow \mathbb{R}^+\setminus\{0\}$, $i=1,\ldots,n$, and a positive number $\overline{a}$ such that for a.e.\ $\x\in\mathbb{R}^n$ and all $s_0\in \mathbb{R},\;$  $\s,\t\in\mathbb{R}^n$ the following inequalities are satisfied:
\begin{equation}\label{us1}
|a_i(\x,s_0,\s)|\leq \widehat{a}_i(|s_0|)\left((\P(\x,\s))^{1/p'_i(\x)}
+\Phi_i(\x)\right),\quad i=1,\ldots,n;
\end{equation}
\begin{equation}\label{us2}
(\a(\x,s_0,\s)-\a(\x,s_0,\t))\cdot(\s-\t)>0,\quad \s\neq \t;
\end{equation}
\begin{equation}\label{us3}
\a(\x,s_0,\s)\cdot\s\geq \overline{a}\P(\x,\s)-\phi(\x).
\end{equation}

Furthermore, assume the existence of a nonnegative function $\Phi_0\in L_{1}(\mathbb{R}^n)$ and a continuous nondecreasing function $\widehat{b}:\mathbb{R}^+\rightarrow \mathbb{R}^+$ such that for a.e.\ $\x\in\mathbb{R}^n$ and all $s_0\in \mathbb{R},\;$ $\s\in\mathbb{R}^n$ there hold
\begin{equation}\label{us4}
|b(\x,s_0,\s)|\leq \widehat{b}(|s_0|)\left(\P(\x,\s)
+\Phi_0(\x)\right);
\end{equation}
\begin{equation}\label{us5}
b(\x,s_0,\s)s_0\geq 0.
\end{equation}

As an example, one can consider functions
$$
a_i(\x,s_0,\s)=\widehat{a}_i(|s_0|)\left(\P(\x,\s)^{1/p'_i(\x)}{\rm sign}\,s_i+\Phi_i(\x)\right)
,\quad i=1,\ldots,n,
$$
$$
b(\x,s_0,\s)=b(s_0)\P(\x,\s)^{1/q'(\x)}\Phi_0^{1/q(\x)}(\x),
$$
with a continuous nondecreasing odd function $b:\mathbb{R}\rightarrow \mathbb{R},\;$
$q\in C^+(\mathbb{R}^n)$, and a nonnegative function $\Phi_0\in L_{1}(\mathbb{R}^n)$.

We will assume that a measure $\mu$ has the form \eqref{us7}.
Denoting
$\widetilde{\a}(\x,s_0,\s)=\a(\x,s_0,\s)+\f$, we obtain from equation \eqref{ur} that
$$-{\rm div}\,\widetilde{\a}(\x,u,\nabla
u)+|u|^{p_0(\x)-2}u+b(\x,u,\nabla
u)=f+f_0.
$$
Applying inequality \eqref{ung}, we easily notice that the function $\widetilde{\a}(\x,s_0,\s)$ also obeys assumptions of the type \eqref{us1}--\eqref{us3}.
Therefore, we will consider equation \eqref{ur} with the measure
\begin{equation}\label{us6}
\mu=f+f_0,\quad f\in L_1(\mathbb{R}^n),\quad f_0\in L_{p'_0(\cdot)}(\mathbb{R}^n).
\end{equation}

Define the function
$T_k(r)=\max(-k,\min(k,r))$
and introduce the notation
$\langle u \rangle=\int\limits_{\mathbb{R}^n}u \, d\x$.
By ${\mathcal{T}}_{\overrightarrow{\bf p}(\cdot)}^1(\mathbb{R}^n)$ we denote the set of measurable functions
$u:$ $\mathbb{R}^n\to\RR$ such that $T_k(u)\in
{W}_{\overrightarrow{\bf p}(\cdot)}^1(\mathbb{R}^n)$ for any $k>0$.
For $u\in {\mathcal{T}}_{\overrightarrow{\bf p}(\cdot)}^1(\mathbb{R}^n)$ and all $k>0$ we have
\begin{equation}\label{3_0}
\nabla T_k(u)=\chi_{\{|u|<k\}}\nabla u \in  \L_{\overrightarrow{\p}(\cdot)}(\mathbb{R}^n).
\end{equation}

\begin{defn}\label{D1}
	An entropy solution of equation	\eqref{ur}, \eqref{us6} is a function
	$u\in {\mathcal{T}}_{\overrightarrow{\bf p}(\cdot)}^1(\mathbb{R}^n)$ such that
	\begin{align*}
	&1)\quad b(\x,u,\nabla u)\in
	L_{1}(\mathbb{R}^n);\nonumber\\
	&2)\quad  \mbox{for all}\; k>0 \;\mbox{and} \; \xi\in
	C_0^1(\mathbb{R}^n) \;\mbox{the following inequality is satisfied:}
	\end{align*}
	\begin{equation}\label{intn}
	\langle (b(\x,u, \nabla u) +|u|^{p_0(\x)-2}u-f-f_0)T_k(u-\xi) \rangle+\langle \a(\x,u,\nabla u)\cdot\nabla T_k(u-\xi)
	\rangle\leq 0.
	\end{equation}
\end{defn}

\begin{defn}\label{D2}
	A renormalized solution of equation \eqref{ur}, \eqref{us6} is a function $u\in {\mathcal{T}}^1_{\overrightarrow{\bf p}(\cdot)}(\mathbb{R}^n)$ such that
	\begin{align*}
	&1)\quad   b(\x,u,\nabla u)\in
	L_{1}(\mathbb{R}^n);\\
	&2)\quad \lim_{h\rightarrow \infty}\int\limits_{\{h\leq |u|<  h+1\}}\P(\x,\nabla u) \, d\x=0;\nonumber\\
	&3)\quad \mbox{for any smooth function}\; S\in W^1_{\infty}(\mathbb{R})\;\mbox{with a compact support} \\
	&\mbox{and any function}\; \xi\in C_0^1(\mathbb{R}^n)\; \mbox{the following equality is satisfied:}
	\end{align*}
	\begin{align}
	\notag
	\langle(b(\x,u,\nabla u) &+|u|^{p_0(\x)-2}u-f-f_0)S(u)\xi\rangle
	\\
	&+\langle \a(\x,u,\nabla u)\cdot(S'(u)\xi\nabla u+S(u)\nabla\xi)\rangle= 0.
	\label{ren}
	\end{align}
\end{defn}

The main results of the present article are the following two theorems.
\begin{thm}\label{t1}
	Let assumptions \eqref{us0}--\eqref{us5} be satisfied.
	Then there exists an entropy solution of equation \eqref{ur},  \eqref{us6}.
\end{thm}

\begin{thm}\label{t2}
	Let assumptions \eqref{us0}--\eqref{us5} be satisfied.
	Then the entropy solution obtained in Theorem \ref{t1} is a renormalized solution of equation \eqref{ur},  \eqref{us6}.
\end{thm}

Previously, statements similar to Theorems \ref{t1}, \ref{t2} on solutions of the Dirichlet problem for equation \eqref{ur} with $\mu\in L_1(\Omega)$ have been proved by the author in \cite{Kozh_2019_ms}, \cite{Kozh_2019}.
Therefore, to avoid a repetition, in the present article we will mainly concentrate on changes caused by the presence of the term $f_0\in L_{p'_0(\cdot)}(\mathbb{R}^n)$ in the right-hand side of equation \eqref{ur}, \eqref{us6}, and on differences caused by the absence of the boundary conditions \eqref{gu}.

\section{Preliminaries}\label{s4}
Applying \eqref{sum}, we obtain from inequalities \eqref{us1} that for a.e.\ $\x\in\mathbb{R}^n$ and all $(s_0,\s)\in\mathbb{R}^{n+1}$ there holds
$$
|a_i(\x,s_0,\s)|^{p'_i(\x)}\leq \widehat{A}_i(|s_0|)
\left(\P(\x,\s)+\Psi_i(\x)\right),\quad
i=1,\ldots,n,
$$
where each $\Psi_i\in L_1(\mathbb{R}^n)$ is a nonnegative function and each $\widehat{A}_i:\mathbb{R}^+\rightarrow \mathbb{R}^+\setminus\{0\}$ is a continuous nondecreasing function.
This implies the inequality
\begin{equation}\label{P's}
\P'(\x,\a(\x,s_0,\s))\leq \widehat{A}^n(|s_0|)(\P(\x,\s)+\Psi^n(\x)),
\end{equation}
where $\widehat{A}^n(|s_0|)=\sum\limits_{i=1}^n\widehat{A}_i(|s_0|)$, $\Psi^n(\x)=\sum\limits_{i=1}^n\Psi_i(\x)$, $\P'(\x,\s)=\sum\limits_{i=1}^n|s_i|^{p'_i(\x)}$.
From \eqref{3_0}, \eqref{P's} we deduce that
\begin{equation}\label{axN}
\chi_{\{|u|< k\}}\a(\x,u, \nabla u)\in
\L_{\overrightarrow{\p}'(\cdot)}(\mathbb{R}^n)
\quad \text{for any} \; k > 0.
\end{equation}

It is assumed that the integral of $|u|^{p_0(\x)-2}u T_k(u-\xi)$ in inequality \eqref{intn} is convergent for any $k>0$. The convergence of the remaining integrals in \eqref{intn} follows from \eqref{us6}, \eqref{3_0}, \eqref{axN}, and assumption 1) of Definition \ref{D1}.
The convergence of all integrals in equality \eqref{ren} follows from \eqref{us6}, \eqref{3_0}, \eqref{axN}, and assumption 1) of Definition \ref{D2}.

All constants which will appear in the sequel are assumed to be positive.

\begin{lem}\label{lemma_1}
	Let $u$ be an entropy solution of equation \eqref{ur},  \eqref{us6}.
	Then for any $k>0$ the following inequality is satisfied:
	\begin{equation}\label{lemma_1_}
	\int\limits_{\{ |u|<k\}}{\bf P}(\x,u,\nabla u)\, d\x+k\int\limits_{\{ |u|\geq k\}}|u|^{p_0(\x)-1} \, d\x\leq C_1 k+C_2.
	\end{equation}
\end{lem}
\begin{proof}
	Inequality \eqref{intn} with $\xi=0$ has the form
	\begin{align*}
	I=\int\limits_{\mathbb{R}^n}(|u|^{p_0(\x)-2}u + b(\x,u,\nabla
	u))T_k(u) \, d\x
	&+\int\limits_{\{ |u|<k\}}\a(\x,u,\nabla
	u)\cdot \nabla u \, d\x \\
	&\leq\int\limits_{\mathbb{R}^n}(f+f_0)
	T_{{k}}(u)\, d\x=J.
	\end{align*}
	
	Applying inequalities \eqref{us3}, \eqref{us5}, we estimate $I$ as
	\begin{equation}
	I\geq k\int\limits_{\{ |u|\geq k\}}|u|^{p_0(\x)-1}\, d\x+\int\limits_{\{ |u|<k\}} |u|^{p_0(\x)}\, d\x+\overline{a}\int\limits_{\{|u|<k\}} \P(\x,\nabla u)\, d\x -\|\phi\|_1.\label{lemma_1(1)}
	\end{equation}
	Using inequality \eqref{ung}, we obtain the following bound for $J$ for any $\varepsilon>0$:
	\begin{align}
	\notag
	&|J|\\
	&\leq \int\limits_{\mathbb{R}^n }|f||T_{k}(u)|\, d\x+C(\varepsilon)\int\limits_{\mathbb{R}^n}|f_0|^{p_0'(\x)}\, d\x+\varepsilon\int\limits_{\{ |u|< k\}}|u|^{p_0(\x)}\, d\x+ \varepsilon\int\limits_{\{|u|\geq k\}}k^{p_0(\x)}\, d\x \nonumber\\
	&\leq k \int\limits_{\mathbb{R}^n}|f|\, d\x+C_3(\varepsilon)+\varepsilon\int\limits_{\{|u|< k\}}|u|^{p_0(\x)}\, d\x+ \varepsilon k\int\limits_{\{|u|\geq k\}}|u|^{p_0(\x)-1}\, d\x.\label{lemma_1(2)}
	\end{align}
	
	Combining \eqref{lemma_1(1)}, \eqref{lemma_1(2)}, and choosing an appropriate $\varepsilon$, we conclude that
	\begin{equation*}
	\frac{k}{2}\int\limits_{\{ |u|\geq k\}}|u|^{p_0(\x)-1}\, d\x+\frac{1}{2}\int\limits_{\{ |u|<k\}} |u|^{p_0(\x)}\, d\x+{\overline{a}}\int\limits_{\{ |u|<k\}} \P(\x,\nabla u)\, d\x \leq
	 k C_4+C_3.
	\end{equation*}
	Thus, the claimed inequality \eqref{lemma_1_} is proved.
\end{proof}

\begin{rem}\label{remark_1}
	If $u$ is an entropy solution of equation \eqref{ur}, \eqref{us6}, then it follows from Lemmas \ref{lemma_1} that
	\begin{equation}\label{remark_1_1}
	{\rm meas}\,\{|u|\geq  k\}\rightarrow 0, \quad k\rightarrow \infty;
	\end{equation}
	$$\forall k>0\quad |u|^{p_0(\x)-1}\chi_{\{|u|\geq k\}}\in L_1(\mathbb{R}^n).
	$$
	Moreover,
	$$
	\forall k>0\quad |u|^{p_0(\x)}\chi_{\{|u|< k\}}\in L_1(\mathbb{R}^n).
	$$
\end{rem}

\begin{lem}\label{lemma_6}
	Let $u$ be an entropy solution of equation \eqref{ur}, \eqref{us6}.
	Then inequality $\eqref{intn}$ is valid for any function $\xi\in {W}_{\overrightarrow{\bf p}(\cdot)}^1(\mathbb{R}^n)\cap L_{\infty}(\mathbb{R}^n)$.
\end{lem}
The proof of Lemma \ref{lemma_6} is similar to the proof of \cite[Lemma 4.5]{Kozh_2019}.

\begin{lem}\label{lemma_3}
	Let $u$ be an entropy solution of equation \eqref{ur}, \eqref{us6}.
	Then for all $k>0$ there holds
	\begin{equation}\label{lemma_3_}
	\lim_{h\to \infty}\int\limits_{\{ h\leq |u|<k+h\}}\P(\x,\nabla u) \, d\x=0.
	\end{equation}
\end{lem}
\begin{proof}
	Taking $\xi=T_{h} (u)$ in inequality \eqref{intn}, we obtain
	\begin{align*}
	&\int\limits_{\{ h\leq|u|<k+h\}} \a(\x,u,\nabla u)\cdot\nabla u \, d\x
	\\
	&+\int\limits_{\{ h\leq|u|\}}\left(b(\x,u,\nabla u)+|u|^{p_0(\x)-2}u\right)T_{k}(u-T_h(u)) \, d\x
	\leq k\int\limits_{\{ h\leq|u|\}}(|f|+|f_0|)\, d\x.\nonumber
	\end{align*}
	Then, using \eqref{us5}, we derive the inequality
	\begin{align*}
	\int\limits_{\{ h\leq|u|<k+h\}} \a(\x,u,\nabla u)&\cdot\nabla u \, d\x + k \int\limits_{\{|u|\geq k+h\}}\left(|b(\x,u,\nabla u)|+|u|^{p_0(\x)-1}\right) \, d\x
	\\
	&+\int\limits_{\{ h\leq|u|< k+h\}}\left(b(\x,u,\nabla u)+|u|^{p_0(\x)-2}u\right)(u-h\,{\rm sign}\,u) \, d\x\\
	&\leq k\int\limits_{\{ h\leq|u|\}}(|f|+|f_0|^{p_0'(\x)})\, d\x+k\,{\rm meas}\{|u|\geq h\}.
	\end{align*}
	In view of \eqref{us5}, the following inequality holds true for all $h\leq |u|$:
	$$
	(b(\x,u,\nabla u)+|u|^{p_0(\x)-2}u)(u-h\,{\rm sign}\,u)\ge 0.
	$$
	Combining the last two inequalities and applying \eqref{us3}, we derive that
	\begin{align*}
	\overline{a}\int\limits_{\{ h\leq |u|<k+h\}}\P(\x,\nabla u) \, d\x
	&\leq
	k\int\limits_{\{ h\leq|u|\}}\left(|f|+|f_0|^{p_0'(\x)}\right)\, d\x\\
	&+\int\limits_{\{h\leq  |u|\}}|\phi|\, d\x+k \, {\rm meas}\{h\leq |u|\}
	\end{align*}
	for any $k>0$.
	Finally, recalling that $f, \phi \in L_1(\Omega)$, $f_0\in L_{p'_0(\cdot)}(\Omega)$, and taking into account \eqref{remark_1_1}, we make the passage to the limit as $h\to \infty$ and obtain \eqref{lemma_3_}.
\end{proof}

\begin{lem}\label{lemma_4}
	Let $v^j,\, j\in \mathbb{N}$, and $v$ be functions from $L_{p(\cdot)}(\mathbb{R}^n)$ such that
	$\{v^j\}_{j\in \mathbb{N}}$  is bounded in
	$L_{p(\cdot)}(\mathbb{R}^n)$  and
	$$
	v^j\rightarrow
	v \quad\mbox{a.e.\ in}\quad \mathbb{R}^n,\quad j\rightarrow\infty.
	$$
	Then
	$$
	v^j\rightharpoonup v
	\quad\mbox{weakly in}\quad L_{p(\cdot)}(\mathbb{R}^n),\quad j\rightarrow\infty.
	$$
\end{lem}

For the space $L_{p(\cdot)}(Q)$, where $Q$ is a bounded domain, the proof of Lemma \ref{lemma_4} is identical with the proof of Lemma 1.3 from \cite[Chapter 1]{Lions} ($p(\x)=p$).
Then, the validity of Lemma \ref{lemma_4} follows from the reflexivity of the space $L_{p(\cdot)}(\mathbb{R}^n)$ and the weak compactness of a bounded set in $L_{p(\cdot)}(\mathbb{R}^n)$.

\begin{lem}\label{lemma_8}
	Let functions $v^j,\, j\in \mathbb{N},$ and $v\in L_{\infty}(\mathbb{R}^n)$ be such that $\{v^j\}_{j\in \mathbb{N}}$ is bounded in  $L_{\infty}(\mathbb{R}^n)$ and
	$$
	v^j\rightarrow v \quad\mbox{a.e.\ in}\quad \mathbb{R}^n,\quad j\rightarrow\infty.
	$$
	Then
	$$
	v^j\stackrel{*}\rightharpoonup v \quad\mbox{weakly in}\quad L_{\infty}(\mathbb{R}^n),\quad j\rightarrow\infty.
	$$
	
	If, in addition, $h^j,\, j\in \mathbb{N}$, and $h$ are functions from $L_{p(\cdot)}(\mathbb{R}^n)$ such that
	$$
	h^j \rightarrow h \quad\mbox{strongly in}\quad L_{p(\cdot)}(\mathbb{R}^n),\quad j\rightarrow\infty,
	$$
	then
	$$
	v^jh^j\rightarrow v h \quad\mbox{strongly in}\quad L_{p(\cdot)}(\mathbb{R}^n),\quad j\rightarrow\infty.
	$$
\end{lem}
The proof of Lemma \ref{lemma_8} follows from the Lebesgue theorem.

\section{Proofs of Theorems \ref{t1}, \ref{t2}}\label{s5}

\begin{proof}[Proof of Theorem \ref{t1}]
	
	\textit{Step 1.}
	Let us construct a sequence $\{f^m(\x)\}_{m=1}^{\infty}$ in the same way as in Lemma \ref{lemma-1}.
	Consider the equations
	\begin{equation}\label{theorem_2(1-2)}
	{\rm div}\;\a^m(\x,u,\nabla u)=a_0^m(\x,u,\nabla u),\quad \x\in \mathbb{R}^n,\quad m\in\mathbb{N},
	\end{equation}
	with functions
	\begin{align*}
	&\a^m(\x,s_0,\s)=\a(\x,T_m(s_0),\s),\\
	&a_0^m(\x,s_0,\s)=b^m(\x,s_0,\s)+|s_0|^{p_0(\x)-2}s_0-f^m(\x)-f_0(\x).
	\end{align*}
	Here
	\begin{align*}
	\a^m(\x,s_0,\s)=(a^m_1(\x,s_0,\s),\dots,a^m_n(\x,s_0,\s)),\quad  b^m(\x,s_0,\s)=T_m(b(\x,s_0,\s))\chi_{ B(m)}.
	\end{align*}
	Clearly,
	\begin{equation}\label{theorem_2(1-3)}
	|b^m(\x,s_0,\s)|\leq |b(\x,s_0,\s)|,\quad |b^m(\x,s_0,\s)|\leq m\chi_{ B(m)},\quad \x\in \mathbb{R}^n,\quad (s_0,\s)\in \mathbb{R}^{n+1}.
	\end{equation}
	Moreover, applying \eqref{us5}, we get the inequality
	\begin{equation}\label{theorem_2(1-4)}
	b^m(\x,s_0,\s)s_0\geq 0,\quad \x \in \mathbb{R}^n,\quad  (s_0,\s)\in \mathbb{R}^{n+1}.
	\end{equation}
	
	Define the operators ${\bf A}^m:{W}^{1}_{\overrightarrow{\bf p}(\cdot)}(\mathbb{R}^n)\rightarrow{W}^{-1}_{\overrightarrow{\bf p}'(\cdot)}(\mathbb{R}^n)$, ${\bf A}^m=\A^m+{A}^m_0$ for any $u,v\in {W}^{1}_{\overrightarrow{\bf p}(\cdot)}(\mathbb{R}^n)$ by the equalities
	\begin{equation*}\label{theorem_2(1-6)}
	<{\A}^m(u),v>=\langle \a^m(\x,u, \nabla u)\cdot\nabla
	v\rangle,\quad
	<{A}^m_0(u),v>=\langle a^m_0(\x,u,\nabla u)v\rangle.
	\end{equation*}
	A generalized solution of equation \eqref{theorem_2(1-2)} is a function $u\in{W}^{1}_{\overrightarrow{\bf p}(\cdot)} (\mathbb{R}^n)$ which satisfies the identity
	\begin{equation*}\label{theorem_2(1-5-6)}
	<{\bf A}^m(u),v>= 0
	\end{equation*}
	for any function $v\in {W}^{1}_{\overrightarrow{\bf p}(\cdot)}(\mathbb{R}^n)$.
	
	The following result holds true.
	\begin{thm}\label{t3}
		Let assumptions \eqref{us0}--\eqref{us5} be satisfied.
		Then there exists a generalized solution of equation \eqref{theorem_2(1-2)}.
	\end{thm}
	Theorem \ref{t3} can be proved in much the same way as \cite[Theorem A.1]{Kozh_2019}.
	
	By Theorem \ref{t3}, for every $m\in\mathbb{N}$ there exists a generalized solution $u^m\in{W}^{1}_{\overrightarrow{\bf p}(\cdot)}(\mathbb{R}^n)$ of equation \eqref{theorem_2(1-2)}.
	Thus, for any function $v\in {W}^{1}_{\overrightarrow{\bf p}(\cdot)}(\mathbb{R}^n)$ the following identity holds true:
	\begin{align}
	\nonumber
	\langle (b^m(\x,u^m,\nabla u^m)&+|u^m|^{p_0(\x)-2}u^m-f^m(\x)-f_0(\x))v\rangle
	\\
	&+\langle \a(\x,T_m (u^m), \nabla u^m)\cdot\nabla
	v\rangle= 0,\quad m\in\mathbb{N}.
	\label{theorem_2(1-5)}
	\end{align}

	\textit{Step 2.}
	In this step, we establish a priori estimates for the sequence $\{u^m\}_{m\in \mathbb{N}}$.
	
	Taking $v=T_{k,h} (u^m)=T_{k}(u^m-T_h(u^m))$, $h>k>0$, in \eqref{theorem_2(1-5)} and applying inequality \eqref{ung}, we obtain
	\begin{align}
	&\int\limits_{\{ h\leq|u^m|<k+h\}} \a(\x,T_m(u^m),\nabla u^m)\cdot\nabla u^m \, d\x
	\nonumber\\
	&+\int\limits_{\{ h\leq|u^m|\}}\left(b^m(\x,u^m,\nabla u^m)+|u^m|^{p_0(\x)-2}u^m\right)T_{k,h} (u^m)\, d\x\label{theorem_2(2-1)}\\
	&\leq k\int\limits_{\{|u^m|\geq h\}}|f^m|\, d\x+C_1(\varepsilon)\int\limits_{\{ |u^m|\geq h\}}|f_0|^{p'_0(\x)}\, d\x+\varepsilon\int\limits_{\{|u^m|\geq h\}}|T_{k,h} (u^m)|^{p_0(\x)}\, d\x.\nonumber
	\end{align}
	Thanks to $\eqref{theorem_2(1-4)}$, the following inequality is satisfied for any $h\leq |u^m|$:
	$$
	(b^m(\x,u^m,\nabla u^m)+|u^m|^{p_0(\x)-2}u^m)T_{k,h} (u^m)\ge 0.
	$$
	Due to this fact, we obtain from \eqref{theorem_2(2-1)} that
	\begin{align*}
	&\int\limits_{\{ h\leq|u^m|<k+h\}} \a(\x,T_m(u^m),\nabla u^m)\cdot\nabla u^m \, d\x
	\\
	&\quad+\int\limits_{\{ h\leq|u^m|<k+h\}}\left(|b^m(\x,u^m,\nabla u^m)|+|u^m|^{p_0(\x)-1}\right)|u^m-h{\rm sign}\, u^m |\, d\x
	\\
	&\quad+k\int\limits_{\{|u^m|\geq k+h\}}\left(|b^m(\x,u^m,\nabla u^m)|+|u^m|^{p_0(\x)-1}\right) \, d\x
	\\
	&\leq k\int\limits_{\{|u^m|\geq h\}}|f^m|\, d\x+C_1(\varepsilon)\int\limits_{\{ |u^m|\geq h\}}|f_0|^{p'_0(\x)}\, d\x
	\\
	&\quad+\varepsilon 2^{p_0^+-1}\int\limits_{\{ h\leq|u^m|<k+h\}}|u^m|^{p_0(\x)-1}|u^m-h{\rm sign}\, u^m |\, d\x
	\\
	&\quad+\varepsilon k \int\limits_{\{ |u^m|\geq k+ h\}}|u^m|^{p_0(\x)-1}\, d\x.\nonumber
	\end{align*}
	Choosing an appropriate $\varepsilon$ and using \eqref{theorem_2(1-1)}, \eqref{us3}, we derive the inequality
	\begin{align}
	\int\limits_{\{ h\leq|u^m|<k+h\}} &\left(\a(\x,T_m(u^m), \nabla u^m)\cdot\nabla u^m +\phi\right)\, d\x\label{theorem_2(2-2)}\\
	&+k\int\limits_{\{|u^m|\geq k+h\}}\left(|b^m(\x,u^m,\nabla u^m)|+\frac{1}{2}|u^m|^{p_0(\x)-1}\right) \, d\x
	\nonumber\\
	&\leq \int\limits_{\{ |u^m|\geq h\}}\left(k|f|+C_1|f_0|^{p'_0(\x)}+\phi\right)\, d\x\leq C_2k+C_3,\quad m\in \mathbb{N}.\nonumber
	\end{align}
	
	Let us now take $T_k (u^m)$ as a test function in \eqref{theorem_2(1-5)}. Performing the same calculations as above, we get the inequality
	\begin{align}
	\overline{a}\int\limits_{\{ |u^m|<k\}}\P(\x,\nabla u^m) \, d\x
	&+ k\int\limits_{\{ |u^m|\geq k\}}\left(|b^m(\x,u^m,\nabla u^m)|+\frac{1}{2}|u^m|^{p_0(\x)-1}\right)  \, d\x\nonumber\\
	&+\frac{1}{2}\int\limits_{\{ |u^m|< k\}}|u^m|^{p_0(\x)} \, d\x\leq C_2k+C_4,\quad m\in \mathbb{N}.\label{theorem_2(2-3)}
	\end{align}
	
	The estimate \eqref{theorem_2(2-3)} implies
	\begin{align}
	&\int\limits_{\mathbb{R}^n}|T_k(u^m)|^{p_0(\x)}\, d\x=\int\limits_{\{ |u^m|<k\}}|u^m|^{p_0(\x)}\, d\x+\int\limits_{\{ |u^m|\geq k\}}k^{p_0(\x)}\, d\x\nonumber\\
	&\leq\int\limits_{\{ |u^m|<k\}}|u^m|^{p_0(\x)}\, d\x+k\int\limits_{\{ |u^m|\geq k\}}|u^m|^{p_0(\x)-1}\, d\x\leq C_5(k),\quad m\in \mathbb{N}.
	\label{theorem_2(2-5)}
	\end{align}
	Moreover, we obtain from \eqref{theorem_2(2-3)} the bound
	\begin{equation}\label{theorem_2(2-6)}
	\int\limits_{\{ |u^m|<k\}}\P(\x,\nabla u^m) \, d\x=\int\limits_{\mathbb{R}^n}\P(\x,\nabla T_k (u^m))\, d\x\leq C_6(k),\quad m\in \mathbb{N}.
	\end{equation}
	
	Combining \eqref{theorem_2(1-3)}, \eqref{us4}, \eqref{theorem_2(2-6)}, we derive the inequality
	\begin{equation}\label{theorem_2(2-7)}
	\int\limits_{\{|u^m|<k\}}|b^m(\x,u^m,\nabla u^m)|\, d\x\leq  \widehat{b}(k)\int\limits_{\{|u^m|<k\}} \left(\P(\x,\nabla T_k
	(u^m))+\Phi_0(\x)\right)\, d\x
	\leq C_7(k)
	\end{equation}
	for all $m \in \mathbb{N}$.
	From \eqref{theorem_2(2-7)}, \eqref{theorem_2(2-3)}, we get the estimate
	\begin{equation}\label{theorem_2(2-8)}
	\|b^m(\x,u^m,\nabla u^m)\|_{1} \leq C_8(k),\quad m\in\mathbb{N}.
	\end{equation}
	
	\textit{Step 3.}
	We have from the estimate \eqref{theorem_2(2-3)} that
	\begin{equation*}\label{theorem_2(3-1)}
	{\rm meas}\,\{ |u^m|\geq\rho\}\rightarrow 0 \quad \mbox {uniformly with respect to}\; m\in\mathbb{N}, \quad \rho\rightarrow \infty.
	\end{equation*}
	
	Let us establish the following convergence up to a subsequence:
	\begin{equation}\label{theorem_2(3-2)}
	u^m\rightarrow u  \quad \mbox {a.e.\ in} \quad \mathbb{R}^n,\quad m\rightarrow \infty.
	\end{equation}
	Let $\eta_R(r)=\min(1,\max(0,R+1-r))$.
	Then, applying \eqref{sum}, we derive from the estimate \eqref{theorem_2(2-6)} that
	\begin{align*}
	&\int\limits_{\mathbb{R}^n}\P(\x,\nabla (\eta_{R}(|\x|)T_\rho (u^m)))\, d\x
	\\
	&\leq C_9 \int\limits_{\{|u^m|<\rho\}}\P(\x,\nabla u^m ) \, d\x
	+C_9\int\limits_{\mathbb{R}^n}\P(\x,T_\rho (u^m)\nabla \eta_{R}(|\x|))\, d\x\leq C_{10}(\rho,R)
	\end{align*}
	for any $R,\rho>0$.
	Thus, for any fixed $\rho, R>0$ we have the boundedness of the set $\{\eta_{R} T_\rho (u^m)\}_{m\in\mathbb{N}}$ in
	$\mathring{H}_{\overrightarrow{\p}(\cdot)}^1( B(R+1))$.
	Thanks to assumption \eqref{us0}, the space $\mathring{H}_{\overrightarrow{\bf p}(\cdot)}^{1}( B(R+1))$ is compactly embedded in the space $L_{p_-(\cdot)}(B(R+1))$ by Lemma \ref{lemma_0}.
	Therefore, for any fixed $\rho, R>0$ we get the convergence
	$T_\rho (u^m) \rightarrow v_{\rho}$ in $L_{p_-(\cdot)}( B(R))$, as well as the convergence up to a subsequence
	$T_\rho(u^m)\rightarrow v_{\rho}$ a.e.\ in $ B(R)$.
	Then, we obtain convergence \eqref{theorem_2(3-2)} in much the same way as in \cite[5.3]{Kozh_2019}.
	
	It follows from convergence \eqref{theorem_2(3-2)} that for any $k>0$,
	\begin{equation}\label{theorem_2(3-7)}
	T_k(u^m)\rightarrow T_k(u)  \quad \mbox {a.e.\ in} \quad \mathbb{R}^n,\quad m\rightarrow \infty.
	\end{equation}
	
	In view of inequality \eqref{theorem_2(2-2)}, the convergence
	\begin{equation}\label{theorem_2(3-8)}
	|u^m|^{p_0(\x)-2}u^m \to  |u|^{p_0(\x)-2}u \quad\mbox{in} \quad L_{1,{\rm loc}}(\mathbb{R}^n),\quad m\rightarrow \infty,
	\end{equation}
	is established by the same arguments as in \cite[5.3]{Kozh_2019}.
	
	\textit{Step 4.}
	Let us show that $T_k (u)\in  {W}_{\overrightarrow{\bf p}(\cdot)}^{1}(\mathbb{R}^n)$ for any $k>0$.
	Combining \eqref{theorem_2(2-5)}, \eqref{theorem_2(2-6)}, \eqref{st3}, we get the following estimate for any fixed $k>0$:
	$$
	\| T_k(u^m)\|_{
		{W}_{\overrightarrow{\bf p}(\cdot)}^{1}(\mathbb{R}^n)}\leq C_{11}(k),\quad m\in\mathbb{N}.
	$$
	The reflexivity of the space ${W}_{\overrightarrow{\bf p}(\cdot)}^{1}(\mathbb{R}^n)$ allows us to extract a weakly convergent in ${W}_{\overrightarrow{\bf p}(\cdot)}^{1}(\mathbb{R}^n)$ subsequence $T_k (u^m)\rightharpoonup v_k$, $m\rightarrow \infty$,
	where $v_k\in {W}_{\overrightarrow{\bf p}(\cdot)}^{1}(\mathbb{R}^n)$.
	The continuity of the natural map ${W}_{\overrightarrow{\bf p}(\cdot)}^{1}(\mathbb{R}^n) \mapsto {\bf L}_{\overrightarrow{\bf p}(\cdot)}(\mathbb{R}^n)$ yields the weak convergences
	$$
	\nabla T_k (u^m)\rightharpoonup \nabla v_{k} \quad \mbox{in}\quad \L_{\overrightarrow{\p}(\cdot)}(\mathbb{R}^n),\quad
	T_k (u^m)\rightharpoonup  v_{k} \quad \mbox{in}\quad L_{p_0(\cdot)}(\mathbb{R}^n), \quad m\rightarrow \infty.
	$$
	
	Using convergence \eqref{theorem_2(3-7)} and applying Lemma \ref{lemma_4}, we get the weak convergence
	$$
	T_k (u^m)\rightharpoonup T_k(u) \quad \mbox{in}\quad L_{p_0(\cdot)}(\mathbb{R}^n),\quad m\rightarrow \infty.
	$$
	This implies the equality $v_k=T_k(u)\in {W}_{\overrightarrow{\bf p}(\cdot)}^{1}(\mathbb{R}^n).$

	\textit{Step 5.}
	The strong convergence
	\begin{equation}\label{theorem_2(4--0)}
	\nabla T_k(u^m)\rightarrow \nabla T_k(u)\quad \mbox{in}\quad  \L_{\overrightarrow{\p}(\cdot),{\rm
			loc}}(\mathbb{R}^n),\quad m\rightarrow \infty,
	\end{equation}
	is obtained in much the same way as in \cite[5.5]{Kozh_2019}.
	Convergence \eqref{theorem_2(4--0)} implies the following convergences up to a subsequence:
	\begin{equation}\label{theorem_2(4-30)}
	\nabla u^m\rightarrow \nabla u \quad  \mbox{a.e.\ in} \quad \mathbb{R}^n,\quad m\rightarrow\infty;
	\end{equation}
	\begin{equation}\label{theorem_2(4-31)}
	\nabla T_k(u^m)\rightarrow \nabla T_k(u) \quad  \mbox{a.e.\ in} \quad \mathbb{R}^n,\quad m\rightarrow\infty.
	\end{equation}
	
	By \eqref{theorem_2(2-6)}, \eqref{P's}, we get the following estimate for any $k>0$:
	$$
	\|\a(\x,T_k(u^m),\nabla T_k(u^m))\|_{\overrightarrow{\p}(\cdot)}\leq C_{12}(k),
	\quad m\in \mathbb{N}.
	$$
	Therefore, by Lemma \ref{lemma_4}, using convergences \eqref{theorem_2(3-7)}, \eqref{theorem_2(4-31)}, we deduce the weak convergence
	\begin{equation}\label{7-2}
	\a(\x,T_k(u^m),\nabla T_k(u^m))\rightharpoonup \a(\x,T_k(u),\nabla T_k(u))\quad\mbox{in}\quad \L_{\overrightarrow{\p}'(\cdot)}(\mathbb{R}^n),\quad m\rightarrow \infty.
	\end{equation}
	
	The continuity of $b(\x,s_0,\s)$ with respect to $(s_0,\s) $ and convergences \eqref{theorem_2(3-2)}, \eqref{theorem_2(4-30)} imply
	\begin{equation}\label{theorem_2(4-32)}
	b^m(\x,u^m,\nabla u^m)\to
	b(\x,u, \nabla u)\quad\mbox{a.e.\ in} \quad \mathbb{R}^n,\quad m\rightarrow\infty.
	\end{equation}
	In view of \eqref{theorem_2(4-32)}, using Fatou's lemma, we derive from the estimate \eqref{theorem_2(2-8)} that
	\begin{equation*}\label{theorem_2(4-33)}
	b(\x,u, \nabla u)\in L_1(\mathbb{R}^n).
	\end{equation*}
	Therefore, assumption 1) of Definition \ref{D1} is satisfied.
	
	Thanks to inequality \eqref{theorem_2(2-2)}, using convergence \eqref{theorem_2(4--0)}, one can prove in the same way as in \cite[5.6]{Kozh_2019} that
	\begin{equation}\label{theorem_2(5-1)}
	b^m(\x,u^m,\nabla u^m)\to
	b(\x,u,\nabla u)\quad\mbox{in} \quad L_{1,{\rm loc}}(\mathbb{R}^n),\quad m\rightarrow \infty.
	\end{equation}
	
	\textit{Step 6.}
	To prove inequality \eqref{intn}, let us take a test function $v=T_k(u^m-\xi)$, $\xi\in C_0^1(\mathbb{R}^n)$, in identity \eqref{theorem_2(1-5)}. We get
	\begin{align}\label{theorem_2(7-1)}
	&\int\limits_{\mathbb{R}^n} \a(\x,T_m(u^m),\nabla u^m)\cdot\nabla T_k(u^m-\xi)\, d\x\\
	&+\int\limits_{\mathbb{R}^n} \left(b^m(\x,u^m,\nabla u^m)+|u^m|^{p_0(\x)-2}u^m-f^m-f_0\right)T_k(u^m-\xi)\, d\x=0.\nonumber
	\end{align}
	Let us demonstrate the passage to the limit as $m\rightarrow\infty$ only in the last integral.
	
	For $k_1>\max\{k,2\|\xi\|_{\infty}\}$, applying \eqref{theorem_2(2-3)}, we obtain the estimate
	\begin{align*}
	&\int\limits_{\mathbb{R}^n}|T_k(u^m-\xi)|^{p_0(\x)}\, d\x=\int\limits_{\{|u^m-\xi|<k_1\}}|T_k(u^m-\xi)|^{p_0(\x)}\, d\x+\int
	\limits_{\{|u^m-\xi|\geq k_1\}}k^{p_0(\x)}\, d\x\\
	&\leq \int\limits_{\{|u^m|<\widehat{k}_1\}}|u^m-\xi|^{p_0(\x)}\, d\x+\int\limits_{\{|u^m|\geq \widetilde{k}_1\}}k^{p_0(\x)}\, d\x\\
	&\leq 2^{p_0^+-1}\int\limits_{\{|u^m|<\widehat{k}_1\}}(|u^m|^{p_0(\x)}+|\xi|^{p_0(\x)})\, d\x+C_{13}\int\limits_{\{|u^m|\geq \widetilde{k}_1\}}\widetilde{k}_1^{p_0(\x)}\, d\x\leq C_{14}
	\end{align*}
	for all $m\in \mathbb{N}$,
	where $\widetilde{k}_1=k_1-\|\xi\|_{\infty}$, $\widehat{k}_1=k_1+\|\xi\|_{\infty}$.
	Therefore, thanks to \eqref{theorem_2(3-2)}, Lemma \ref{lemma_4} gives the weak convergence
	\begin{equation}\label{theorem_2(7-9)}
	T_k(u^m-\xi)\rightharpoonup T_k(u-\xi),\quad m\to
	\infty, \quad\mbox{in} \quad L_{p_0(\cdot)}(\Omega).
	\end{equation}
	Since $f_0\in L_{p'_0(\cdot)}(\Omega)$, convergence \eqref{theorem_2(7-9)} implies
	$$
	\int\limits_{\Omega}f_0 T_k(u^m-\xi)\, d\x\rightarrow\int\limits_{\Omega}f_0 T_k(u-\xi)\, d\x,\quad m\to\infty.
	$$
	The passages to the limit in the remaining integrals in equality \eqref{theorem_2(7-1)} are performed as in \cite[5.7]{Kozh_2019}.
\end{proof}

\begin{proof}[Proof of Theorem \ref{t2}]
	Let us prove that the entropy solution satisfies the properties of a renormalized solution.
	Assumption 1) holds true since it coincides with assumption 1) of Definition \ref{D1}.
	Assumption 2) is also satisfied (see \eqref{lemma_3_}).
	
	Let us prove equality \eqref{ren}.
	Let $\{u^m\}_{m\in \mathbb{N}}$ be a sequence of weak solutions of equation \eqref{theorem_2(1-2)}.
	Assume that $S\in W^1_{\infty}(\mathbb{R})$ is such that ${\rm supp}\;S\subset[-M,M]$ for $M>0$.
	For any $m\in \mathbb{N}$ and any function $\xi\in C_0^1(\mathbb{R}^n)$, taking $S(u^m)\xi\in\mathring
	{W}_{\overrightarrow{\bf p}(\cdot)}^1(\mathbb{R}^n)$ as a test function in \eqref{theorem_2(1-5)}, we derive
	\begin{align}
	&\langle \a(\x,T_m(u^m), \nabla u^m)\cdot(S'(u^m)\xi\nabla u^m+S(u^m)\nabla\xi)\rangle \label{7-1}\\
	&+\langle (b^m(\x,u^m,\nabla u^m)+|u^m|^{p_0(\x)-2}u^m-f^m(\x)-f_0)S(u^m)\xi\rangle=I^m+J^m=0.\nonumber
	\end{align}
	
	Clearly,
	\begin{align*}
	I^m
	&=\int\limits_{\mathbb{R}^n}\a(\x,T_m (u^m), \nabla u^m)\cdot(S'(u^m)\xi\nabla u^m+S(u^m)\nabla\xi)\, d\x
	\\
	&=\int\limits_{\mathbb{R}^n}\a(\x,T_M(u^m), \nabla T_M(u^m))\cdot(S'(u^m)\xi\nabla T_M (u^m)+S(u^m)\nabla\xi)\, d\x,~~ m\geq M.
	\end{align*}
	
	Since ${\rm supp}\,\xi$ is a bounded set in $\mathbb{R}^n$, convergences \eqref{theorem_2(3-2)},  \eqref{theorem_2(4--0)}, and Lemma \ref{lemma_8} yield
	\begin{equation}\label{7-3}
	S'(u^m)\xi\nabla T_M(u^m)+S(u^m)\nabla\xi\rightarrow S'(u)\xi\nabla T_M(u)+S(u)\nabla\xi\quad\mbox{in}\quad \L_{\overrightarrow{\p}(\cdot)}(\mathbb{R}^n)
	\end{equation}
	as $m \to \infty$.
	Combining \eqref{7-2}, \eqref{7-3}, we have
	\begin{align}
	\lim_{m\to \infty} I^m
	&=\int\limits_{\mathbb{R}^n} \a(\x,T_M(u),\nabla T_M(u))\cdot(S'(u)\xi\nabla T_M (u)+S(u)\nabla\xi)\, d\x
	\nonumber
	\\
	&=\int\limits_{\mathbb{R}^n}\a(\x,u,\nabla u)\cdot(S'(u)\xi\nabla u+S(u)\nabla\xi)\, d\x.
	\label{7-4}
	\end{align}
	
	Lemma \ref{lemma_8} implies that
	$$
	S(u^m)\xi\stackrel{*}\rightharpoonup S(u)\xi\quad \text{in} \quad L_{\infty}(\mathbb{R}^n),\quad m\rightarrow \infty.
	$$
	Thus, in view of convergences \eqref{theorem_2(1-0)}, \eqref{theorem_2(3-8)},  \eqref{theorem_2(5-1)} and the fact that $f_0\in L_1({\rm supp}\,\xi)$, we obtain the equality
	\begin{equation}\label{7-5}
	\lim_{m\to \infty} J^m=\int\limits_{\mathbb{R}^n}(b(\x,u,\nabla u)+|u|^{p_0(\x)-2}u-f-f_0)S(u)\xi \, d\x.
	\end{equation}
	Combining \eqref{7-1}, \eqref{7-4}, \eqref{7-5}, we get \eqref{ren}.
	Therefore, we conclude that $u$ is a renormalized solution of equation \eqref{ur}, \eqref{us6}.
\end{proof}


\end{document}